\documentclass[10pt]{amsart}
\usepackage[margin=1.2in]{geometry}
\usepackage{colortbl}
\usepackage{color}
\usepackage{cite}

\usepackage{kotex}

\newtheorem{theorem}{Theorem}[section]
\newtheorem{lemma}[theorem]{Lemma}

\theoremstyle{definition}

\theoremstyle{remark}
\newtheorem{remark}[theorem]{Remark}

\numberwithin{equation}{section}

\usepackage{graphicx}
\graphicspath{{./figures/}}
\usepackage{subcaption}
\usepackage{multirow}
\usepackage{bbm}
\usepackage{lipsum}
\usepackage{amsfonts}
\usepackage{graphicx}
\usepackage{epstopdf}
\usepackage{algorithmic}
\usepackage{subcaption}
\usepackage{amssymb}
\usepackage{mathtools}
\usepackage{bbm}

\begin{document}
	\title[]{Mathematical modeling of trend cycle: Fad, Fashion and Classic}

	\author[Bae]{Hyeong-Ohk Bae}
	\address[Bae]{Department of Financial Engineering, Ajou University, Suwon, Republic of Korea}
	\email{hobae@ajou.ac.kr}
	\author[Cho]{Seung Yeon Cho}
	\address[Cho]{Department of Mathematics and Research Institute of Natural Science, 
		Gyeongsang National University, Jinju, Republic of Korea}
	\email{chosy89@gnu.ac.kr}
	\author[Yoo]{Jane Yoo}
	\address[Yoo]{Department of Financial Engineering, Ajou University, Suwon, Republic of Korea}
	\email{janeyoo@ajou.ac.kr}
	\author[Yun]{Seok-Bae Yun}
	\address{Department of mathematics, Sungkyunkwan University, Suwon 440-746, Republic of Korea }
	\email{sbyun01@skku.edu}
	
	\begin{abstract}
	In this work, we suggest a system of differential equations that quantitatively models the formulation and evolution of a trend cycle through the consideration of underlying dynamics between the trend participants. Our model captures the five stages of a trend cycle, namely, the onset, rise, peak, decline, and obsolescence. It also provides a unified mathematical criterion/condition to characterize the fad, fashion and classic. We prove that the solution of our model can capture various trend cycles. Numerical simulations are provided to show the expressive power of our model.   
	\end{abstract}
	\maketitle

\section{Introduction}


People have a tendency to follow or mimic others in a human society, which can be described as a fashion or trend. 
A trend arises in many markets: in a financial market (Tulip mania \cite{V}, Bitcoin frenzy \cite{L}, etc.), a clothing industry (leggings in fast fashion brands \cite{CS}), 
a food industry (Organic food product trend \cite{R}), an entertainment industry (Hallyu \cite{K}), and in an electronic device market (Smartphones becoming fashion \footnote{https://sites.bu.edu/cmcs/2021/10/18/smartphones-fashion-galaxy-iphone/}). 
Such trends, upto some inevitable oversimplification, go through a cycle consisting of onset, rise, peak, decline, and obsolescence. 
A trend emerges 
as a small number of early adopters take the trend (onset). More people start to 
join the trend (rise). After the number of adopters reaches its peak (peak), some begin to leave the trend losing their interest in it (decline). Eventually, the trend 
loses its influencing power and is finally forgotten (obsolescence). These stages constitute a trend cycle. (See Figure \ref{fig classification11} from \cite{kaiser})

\begin{figure}[htbp]
	\centering
	\includegraphics[width=0.6\linewidth]{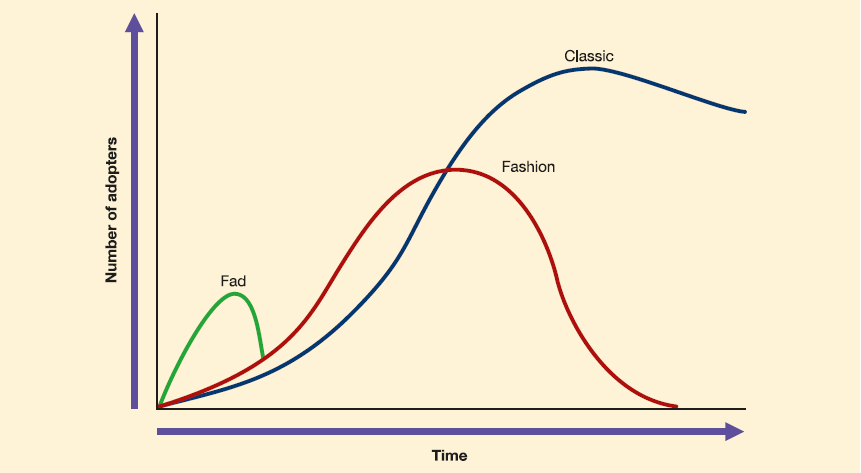}\centering
	\caption{Classification of a trend: Fad, Fashion, and Classic (Source: Susan Kaiser, The Social Psychology of Clothing\cite{kaiser})} \label{fig classification11}
\end{figure}

Such trend cycles are often categorized into Fad, Fashion and Classic. Fad usually refers to a short-lived trend: 
it gains large popularity over a short period of time, and disappears quickly. 
A trend that lasts longer and 
is taken by a larger population is called Fashion. Classic is a style that has firmly settled in a  society or the market. It persists over a much more extended time period. Usually, it enjoys more royalty from customers and is less volatile. But the precise definitions of Fad, Fashion and Classic are not available in the literature (e.g.\cite{B-M, Suk, Y}). Besides the facts that Fad lasts shorter than Fashion, and Classic lasts longer than Fashion, there exists no clear feature that plays as a dividing pole between these three different cycles. For example, it may depend on the business. The usual lifespan of Fad (e.g. Squid Game, Viral video clips on YouTube or TikTok, and Psy's popular song Gangnam Style) in an entertainment business can be much shorter than that of Fad in the fashion industry or electric device industry (e.g. Blackberry and portable media player). Moreover, Fashion may arise as a renovated form of an existing fashion style (e.g. various styles of jeans). Some products of brands such as Coca-Cola, Louis Vuitton, and Chanel are considered as examples of classical items. 

The purpose of this work is twofold. First, we provide a dynamic modeling of the trend cycle. 
We introduce a system of differential equations that, through the consideration of underlying dynamics between the market participants, explains the formation of the trend cycle. We adopt a simple picture that a society is divided into three groups: people who have not joined a particular trend yet, people who have adopted the trend, and people who have left the trend. We then recognize the analogy between the infection/recovery processes in an epidemic spread, and the adoption/rejection processes in a trend cycle to derive dynamic laws between the three groups.

Secondly, based on this model, we provide a mathematically rigorous criterion to characterize Fad, Fashion, and Classic. 
As mentioned above, a trend cycle is dubbed fad, fashion, and classic based on the volatility of the trend. But a clear guideline to characterize them does not exist. In this regard, we observe that the fad, fashion, and classic can be divided by the way how participants leave the trend. As we can see in Figure \ref{fig classification11}, this is reflected in the steepness of the slope in the trend cycle at the stages of decline and obsolescence. This, in turn, can be realized in our differential model by a suitable choice of rejection rates and parameters (see Sec. 2). Our description of a trend cycle provides a rigorous way to characterize these three different types of trends in a precise way. We define a fad as a cycle that extincts in a finite time, leading to hard landing of the curve in a finite time. Classic can be defined by the curve declining at most polynomial order in the long run. Fashion is defined as a trend cycle that lies between them. The seasonal revival of a fashion trend can also be captured. Our trend cycle model provides a unified framework to understand various different types of trends: a fashion trend, a literature style, a commodity's popularity, a cultural frenzy in music and movie, a speculative behavior in a cryptocurrency market.

Studies on a trend cycle involve numerous topics such as marketing, legal issues, consumer behavior, technology, media, historical analysis, cultural study, etc. An exhaustive review on such a broad topic is not plausible. We focus on the literature that is directly relevant to ours. In \cite{P}, a trend cycle model is developed to study how a new design is created, and its popularity eventually falls over time as it spreads across the population. The advent of various social media significantly impacts the fashion trend these days. In \cite{M}, the author examines the relationship between social media and fashion. \cite{KK} investigates the impacts of social media marketing on customers' intimacy and trust in the luxury brand. On the other hand, a recent study in \cite{CL} shows that consumers' emotion becomes a key factor in determining the fashion trend. In a seminal work \cite{Suk}, Hemphill and Suk described the evolution of a fashion trend by combining the effect of flocking and differentiation and investigated various legal issues in the fashion industry based on such an observation. The study on the spread of a fashion trend between social classes has been studied which can be categorized roughly into the following three types: trickle-down theory (spread of a fashion from upper to lower class), trickle-up theory (spread of a fashion from lower to upper class), and trickle-across theory (horizontal movement of fashion) \cite{Sim,Suk,Mac}. See \cite{EM, H, kaiser, King,Suk} and references therein for an overview of various topics in fashion.



As mentioned above, our model has analogy with the epidemic spread model (e.g. susceptible-infected-recovered (SIR) type models). Therefore, a brief review of SIR-type models is in order. Since the inception of the model in \cite{KM}, the SIR model has been widely used for quantitative modeling of epidemics \cite{B-W-A-C,H-L-M,K-S1,K-S2}. The SIR model is also successfully employed to study the recent Covid-19 pandemic to understand the spread mechanism and forecast the possible progress of the pandemic \cite{G-L,L-P-G,A-B-D,N-D-V-H}.

The paper is organized as follows: in Sec. 2, we derive a system of differential equations that describes the dynamics of a trend cycle.  In Sec. 3, we verify that the solution of our model satisfies the desired decaying properties. In Sec. 4, we carry out various numerical experiments to justify our model. In Sec 5, we draw conclusions of the paper.

\section{Dynamic modelling of trend cycle}
We derive a system of differential equations to describe a trend cycle. For this, we introduce three variables that correspond to the number of potential adopters of a trend,  the number of people who adopted the trend, and the number of people who left the trend. We then derive a dynamic law between these three groups by observing how they adopt a trend and how they reject it.


First, we  define the three dynamic variables:
\begin{itemize}
\item $S(t)$: the number of potential adopters of a trend at time $t$,
\item $I(t)$: the number of people who have adopted the trend at time $t$,
\item $R(t)$: the number of people who have left the trends at time $t$. 
\end{itemize}
Note that, by definition, any reasonable model must satisfy 
\begin{align}\label{conservation}
N=S(t)+I(t)+R(t)\quad t\geq 0,
\end{align}
where $N$ is the number of individuals in a society. We now consider how these compartments in the trend interact through the adoption and rejection.\\

\noindent $\bullet$ {\bf Equation for the evolution of $I$: }We start with the derivation of dynamic equation for $I$. The equations for other variables follow almost automatically once the equation for $I$ is determined. For this, we consider two factors below: the adoption, and rejection of a trend.\\

\noindent {\bf Trend Adoption:}
Our main assumption is that the adoption rate of a new trend depends on how often people in a society are exposed to the trend. 
Here, we assume such exposure is expressed in the following multiplicative law:
\begin{align}\label{sum}
\alpha(t) \frac{I(t)}{N} S(t)
\end{align}
for a suitable adoption rate $\alpha$ at each time $t$. 
It remains to determine $\alpha(t)$. For this, we make the following reasonable assumption on the adoption rate: 1) as the ratio of people adopting the trend increases, people get more interested in the trend (that is, $\alpha$ is a monotonic function of $I/N$); 2) As the number of people adopting the trend starts to decrease, people get less interested in the trend.
A suitable adoption rate satisfying these assumptions is a sigmoid function: 
\begin{align}\label{a mono}
\begin{split}
\alpha(t):=\alpha\left(\frac{I(t)}{N}\right)= \frac{m_1}{1+\exp\left(-m_2 \left(\frac{I(t)}{N}-\ell_{\alpha}\right)\right)},
\end{split}
\end{align}
where $m_1>0$ and $m_2>0$ describe the intensity of the adoption and the sharpness of the transition, and $\ell_{\alpha}$ is the adoption delay.\\

\noindent {\bf Trend Rejection:}
To describe how people reject a trend, we employ the following simple expression, which says that the number of people leaving the trend is proportional to $I$:  
\[
\beta(t) I(t).
\]
Our main assumption in choosing the appropriate $\beta(t)$ is that people barely leave the trend when the trend is rising, but start to leave the trend after the trend hits the peak.
It is how quickly people lose their interest after the number of trend-followers gets saturated. 
For this, we let $t=t_*$ to be the first time on which $I$ takes the maximum value, which will be called the transition:
\[
t_*:=\inf\left\{t>0 \hspace{0.2mm}\vert \hspace{0.2mm}\, I'(t)=0.\right\}
\]
Then, we define the rejection rate  as follows: $(p\in\mathbb{R})$

\begin{align}\label{b mono}
\beta(t):=\beta\left(\frac{I(t)}{N}\right)=\left\{ 
\begin{array}{ll}
\frac{m_3}{1+\exp\left(-m_4 \left(\frac{I(t)}{N}-\ell_{\beta}\right)\right)}& ~(t\leq t_*),\\
 C_*\left(\frac{I(t)}{N}\right)^p&~(t>t_*),
\end{array}\right.
\end{align}
where $C_*$ denotes $$C_*=\beta(t_*)\left(\frac{I(t_*)}{N}\right)^{-p}.$$
We note that $m_3, m_4>0$ describe the intensity of rejection and the sharpness of the transition, and $\ell_{\beta}\in \mathbb{R}$ is the delay in rejection.
Note also that there are cases when $I$ decreases from $t=0$. We can set $t_*$ to be infinity in this case. 

The intuition behind this choice of $\beta$ is as follows: Since $I$ decreases after $t=t_*$, the case $p\geq 0$ corresponds to the case when the rate of rejection is slowing down leading to the tail
that lasts forever. On the other hand, $p<0$
corresponds to the case when the rate of rejection accelerates leading to the obsolescence of the trend in finite time, say $T_*>t_*$. 
We will consider this issue further below.


Finally, the rate of change of $I$ is determined by the discrepancy between the adoption and rejection rates:
\begin{align}\label{mono model I}
\begin{split}
I^{\prime}&=\frac{\alpha(I/N)}{N} I S - \beta(I/N) I.
\end{split}
\end{align}

If there is no confusion, we omit $t$ in the rest of this paper for short. 

\noindent $\bullet$ {\bf Equations for the evolution $S$ and $R$: }
We now turn to the equation of $R$. The number of people who found a trend is not interesting anymore is given by the number of people who leave the trend:
 
\begin{align*}
\begin{split}
\beta(I/N) I.
\end{split}
\end{align*}
But some of people who left the fashion become potential consumers again with the ratio $\delta$:
\[
\delta R.
\]
Therefore, the equation for $R$ is presented by 

\begin{align*}
\begin{split}
R^{\prime}= \beta(I/N) I-\delta R.
\end{split}
\end{align*}
Similarly, we have

\begin{align*}
\begin{split}
S^{\prime}=-\frac{\alpha(I/N)}{N}  I S+\delta R.
\end{split}
\end{align*}
\noindent$\bullet$ {\bf System for the trend cycle model:} 
We then combine these three equations to obtain the following dynamic model for a trend:
\begin{align}\label{mono fashion model}
\begin{split}
S^{\prime}&=-\frac{\alpha(I/N)}{N}  I S+ \delta R,\cr
I^{\prime}&=\frac{\alpha(I/N)}{N} I S - \beta(I/N) I,\cr
R^{\prime}&= \beta(I/N) I - \delta R,
\end{split}
\end{align}
where the adoption rate $\alpha$ and the rejection rate $\beta$ are defined in \eqref{a mono} and \eqref{b mono}.\\

\noindent$\bullet$ {\bf Rescaled trend model:} For simplicity, we rewrite 
\[
\frac{S}{N}\rightarrow S,\quad \frac{I}{N}\rightarrow I,\quad \frac{R}{N}\rightarrow R, 
\] 
so that \eqref{mono fashion model} is transformed into the following rescaled form:
\begin{align}\label{main model scale}
\begin{split}
S'&=-\alpha(I) I S+ \delta R,\cr
I'&=\alpha(I) I S - \beta(I) I,\cr
R'&= \beta(I)  I - \delta R.\cr
\end{split}
\end{align}
Note that after the transition time $t_*$, the system becomes
\begin{align}\label{main model scale 2}
	\begin{split}
		S'&=-\alpha(I) I S+ \delta R,\cr
		I'&=\alpha(I) I S - C_*I^{p+1},\cr
		R'&= C_*I^{p+1} - \delta R.\cr
	\end{split}
\end{align}
Throughout this paper, we work on this rescaled system.\\

\noindent$\bullet$ {\bf Characterization of trend cycles - Fad, Fashion and Classic:} 
Now, we explain how our model can characterize three different types of trends, namely, the fad, the fashion, and the classic, by proper choices of $p$ in the rejection rate $\beta$. 
We attempt to provide  $\beta$ with clear mathematical reasoning.
Since we are interested in the behavior after the trend reaches its peak, we consider
only  $t>T_*$, for which the equation for $I$ is (See \eqref{main model scale 2})
$$
I'=\alpha SI-C_* I^{p+1}.
$$
 In the case $p=0$, the equation for $I$ becomes
\begin{align*}
	I'=\alpha SI-C_* I.
\end{align*}
Since $\alpha$ has a lower and upper bound, we can expect that $I$ will behave like
the usual SIR model, for which $I$ decays exponentially.
When $p>0$, $I^{p+1}$ is much smaller than $I$ (since $I$ is normalized: $0\leq I< 1$).
Therefore, we can expect that the decay rate of $I$ will be much slower than that of $p=0$.
Finally, in the case $p<0$, $I^{p+1}$ gets bigger as $I$ decreases to $0$. Moreover, since
$SI$ is bounded by $I$, we expect that $C_* I^{p+1}$ will dominate over $\alpha SI$ so that 
the equation for $I$ in this case is governed by  
$$I'\approx- C_*I^{p+1},$$
which vanishes in a finite time  (see Fig. \ref{fig I_c}).
This intuition leads to the following classification of the phase of $I$: 
\begin{itemize}
\item $I$ vanishes in a finite time with a hard landing if $p\le-1$,
\item $I$ vanishes in a finite time with a soft landing if $-1<p<0$,
\item $I$ decreases exponentially without vanishing in a finite time if $p=0$,
\item $I$ decreases in a polynomial order if $p>0$,
\end{itemize}
which will be verified analytically in Sec. 4, and numerically in Sec. 5. 
This provides us a way to draw a line between the fad, the fashion and the classic in a mathematically rigorous manner:
\begin{itemize}
\item Fad: The cycle experiences a finite time extinction with a hard landing if $p\le -1$.
\item Fast fashion: The cycle experiences a finite time extinction with a soft landing if $-1<p<0$. 
\item Fashion: The cycle declines exponentially fast if $p=0$.
\item Classic: The cycle declines in a polynomial order if $p>0$.
\end{itemize}

This characterization matches well with our understanding that fad, fashion and classic are basically determined by the manner
in which people reject a trend. This, in turn, is reflected in the slope's steepness at the stage of decline and obsolescence, and the duration of the trend cycle. 
We also define th eperiodic case by
\begin{itemize}
\item Periodic: $p\geq 0$ with $\delta>0$.
\end{itemize}
Note that we only consider the case $p\geq0$ for periodic case to prevent the situation where the trend cycle extincts before it enters the next cycle. 
\begin{remark}
 We notice here the equation of $I$ consists of power-law.
 We may find a similarity between the behavior of the trend cycle and the motion of the non-Newtonian power-law fluid.
 It is shown in \cite[Theorem 5.1]{BHO} that
 \begin{itemize}
 \item the energy of the fluid vanishes in a finite time  for $p<0$ (shear thinning fluid),
which corresponds to Fad or Fast fashion,
\item it decays exponentially for $p=0$ (Newtonian fluid, Navier-Stokes equations), which corresponds to Fashion, and
\item it decays polynomially for $p>0$ (shear thickening fluid), which corresponds to Classic.
\end{itemize}
\end{remark}

\begin{figure}[!htb]
	\begin{minipage}{0.48\textwidth}
		\includegraphics[width=1\linewidth]{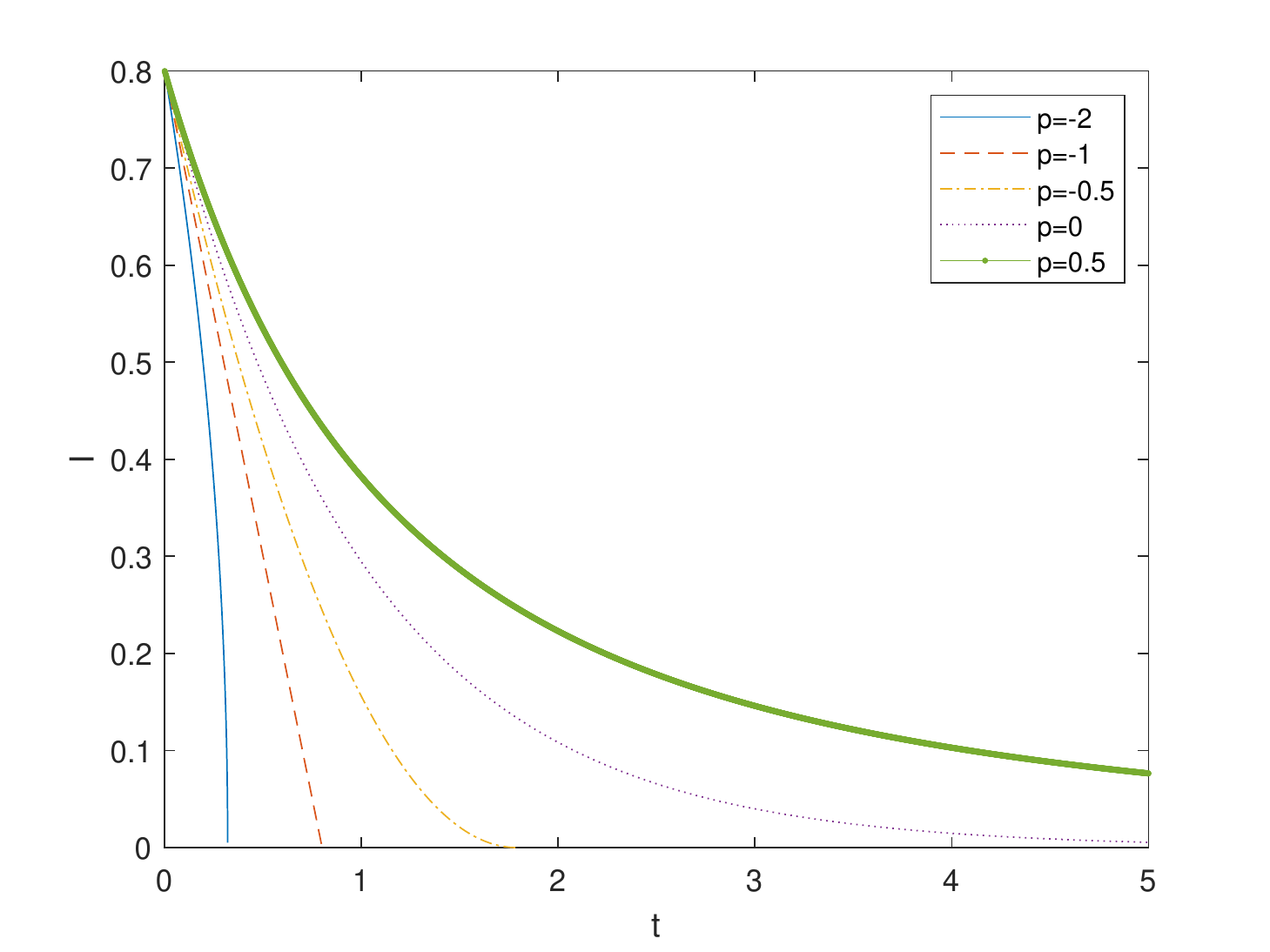}
		\caption{$p$-dependency of $I$. Here we solve the problem $I'(t)=-I^{p+1}$ with $I(0)=0.8$.}\label{fig I_c}
	\end{minipage}\hfill
	\begin{minipage}{0.48\textwidth}
		\includegraphics[width=1\linewidth]{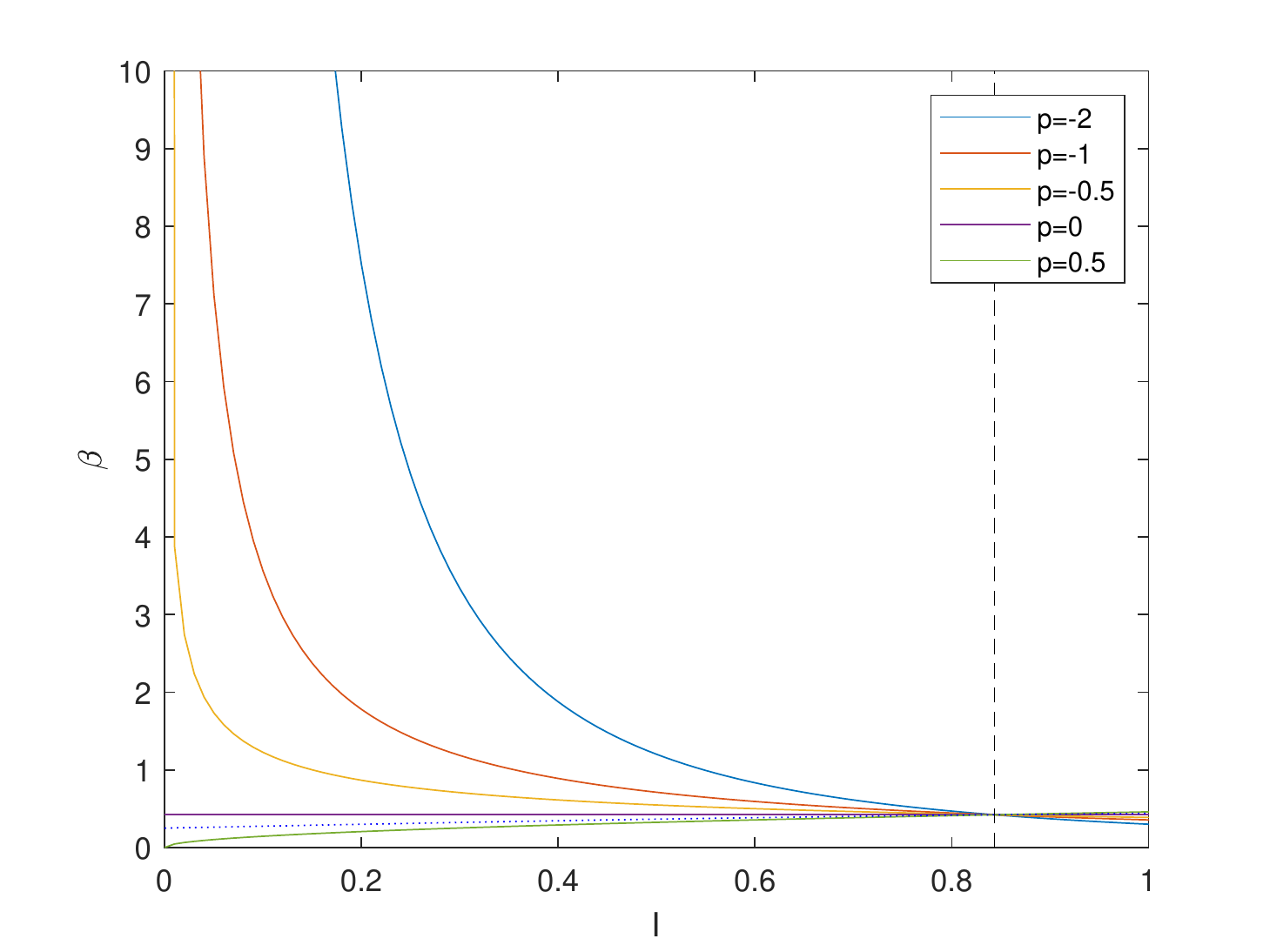}
		\caption{Shape of function $\beta$ after $t=t_*$. Here $I\approx0.84$ is the value of $I(t_*)$.}\label{fig beta}
	\end{minipage}
\end{figure}

\section{Analysis of Classic, Fashion and Fad}
In Section 2, we characterized the three different trends, Fad, Fashion and Classic, based on how the such a trend extinct as time goes.In this section, we verify them by deriving various decay estimates of $I$ depending on $p$. Before we state our main theorem, we need several technical lemmas. Throughout this section,  we only consider the non-periodic case: $\delta=0$. We also  recall that $m_1$ and $m_3$ denote the intensity of adopter/rejection in \eqref{a mono} and \eqref{b mono} respectively. For clarity of the proof, we fix $\ell_{\alpha}=\ell_{\beta}=0$ throughout this section. 
\begin{lemma}\label{conservation lemma}
Let
$$
S(0)+I(0)+R(0)=1.
$$ 	
Then we have
$$
S(t)+I(t)+R(t)=1.
$$
for $t\geq0$. 
\end{lemma}
\begin{proof} 	
We sum up the equations in the system \eqref{main model scale} and observe that the right hand side of \eqref{main model scale} cancels out each other. Then, we get	
$$ \frac{d}{dt}\left(S+ I + R\right) = 0,$$
which implies the desired result.
\end{proof}


In the following lemma, we consider the positivity of the solutions $S(t),~I(t),~R(t)$. 
\begin{lemma}\label{Lem positivity}\emph{\bf[Positivity of solutions]}
 Assume that
	\begin{align}\label{Lem init cond}
S(0) > 0,\quad  I(0)> 0,\quad R(0)= 0.
\end{align}
	Then, we have
	\begin{align*}
			S(t)>0,\quad  I(t)> 0,\quad R(t)> 0
		\end{align*}
	for  $0<t<T_*$.
\end{lemma}
\begin{proof}
	See Appendix.
	\end{proof}

In the following lemma, we provide a sufficient condition under which the transition time $t_*$ becomes finite.

\begin{lemma}\label{Lem t_*}
	 Assume 
		\begin{align*}
	S(0) > 0,\quad  I(0)> 0,\quad R(0)= 0.
	\end{align*}
	Suppose further that  
	\begin{align*}
		\displaystyle m_1S(0)>2m_3.
	\end{align*}
	Then the transition time $t_*>0 $ is finite.
\end{lemma}
\begin{lemma}\label{Prop small} \emph{\bf[Boundedness after $t_*$: $p<0$]} Let $p<0$. 	Assume \[
	S(0),~ I(0),~ R(0)>0
	\] 
	and
	\begin{align*}
		\displaystyle m_1S(0)>2m_3.
	\end{align*}
	Then, there exists $t_*\leq \tau< T_*$ such that   
	$$I\leq \left(\frac{C_*}{2m_1}\right)I^{p+1} ~\mbox{ for all }t>\tau.$$
\end{lemma}
\begin{proof}
	See Appendix.
\end{proof}	

We now study two Bernouli type inequalities
which will be crucially used in the analysis of
long time behavior of the trend cycles.

\begin{lemma}\label{lem bernoulli p>0} \emph{\bf [Bernoulli type differential inequality $(p>0)$]}
	Set $p>0$. Assume that $I(t)>0$ and the following differential inequality holds for $T_0<t<T_1$:  
	\begin{align*}
	\begin{split}
	&- bI^{p+1}	\leq I'\leq aI- bI^{p+1}.
	\end{split}
	\end{align*}
	Then, we have
	\begin{align*}
	\begin{split}
	(I(T_0))^{-p}e^{-pa(t-T_0)} +  \frac{b}{a}\left(1-e^{-pa(t-T_0)}\right)\leq (I(t))^{-p}\leq pb(t-T_0)+(I(T_0))^{-p}. 
	\end{split}
	\end{align*}
\end{lemma}

%
%
%
%
%

\begin{theorem}\label{main theorem1}
	Let $\delta=0$.	Assume
	\[
	S(0),~ I(0),~ R(0)>0
	\] 
	and
	\begin{align*}
		\displaystyle m_1S(0)>2m_3.
	\end{align*}
	Then we have
	\begin{enumerate}
		\item {\bf Classic:}  $(p>0):$ We have for $t>t_*$
		\begin{align*}
			\begin{split}
				\frac{1}{\left(a_{*}^{-1/p}+b(t-t_*)\right)^{\frac{1}{p}}} \leq I(t) &\leq 
				\frac{1}{\left(a_*^{-1/p}-  \left(a_*-c\right)\left(1-e^{-pm_1(t-t_*)}\right)\right)^{\frac{1}{p}}}.
			\end{split}\cr
		\end{align*} 
		\item {\bf Fashion:} $(p=0)$. 	There exists  $t_2>t_*$ such that, for  $ t>t_{*}$
	\begin{align*}
		\begin{split}
			a_* e^{-C_*(t-t_*)}\leq I(t) \leq a_2e^{-\varepsilon(t-t_{2})}.\cr
		\end{split}
	\end{align*}
	\item {\bf Fad:} $p<0$, we have for  $\tau<t\leq T_*$
	\begin{align*}
		\frac{1}{\left(a_\tau^{-p} + b(t-\tau)\right)^{\frac{1}{p}}} \leq I(t)
		&\leq \frac{1}{\left(a_\tau^{-p} + \frac{b}{2}(t-\tau)\right)^{\frac{1}{p}}}.
	\end{align*}
	 
	\end{enumerate}
Here,	$a_*=I(t_*)$,  $a_2=\max\{a_*, I(t_2)\}$, $a_{\tau}=I(\tau)$,   $b=pC_*$,  $c=C_*/m_1$, 	$\varepsilon=\alpha(I(t_{2}))S(t_{2})- C_*$.
	The definition of $t_2$ is given in the proof.
	
\end{theorem}

\begin{figure}[h]
	\centering
	\begin{subfigure}[h]{0.325\linewidth}
		\includegraphics[width=1\linewidth]{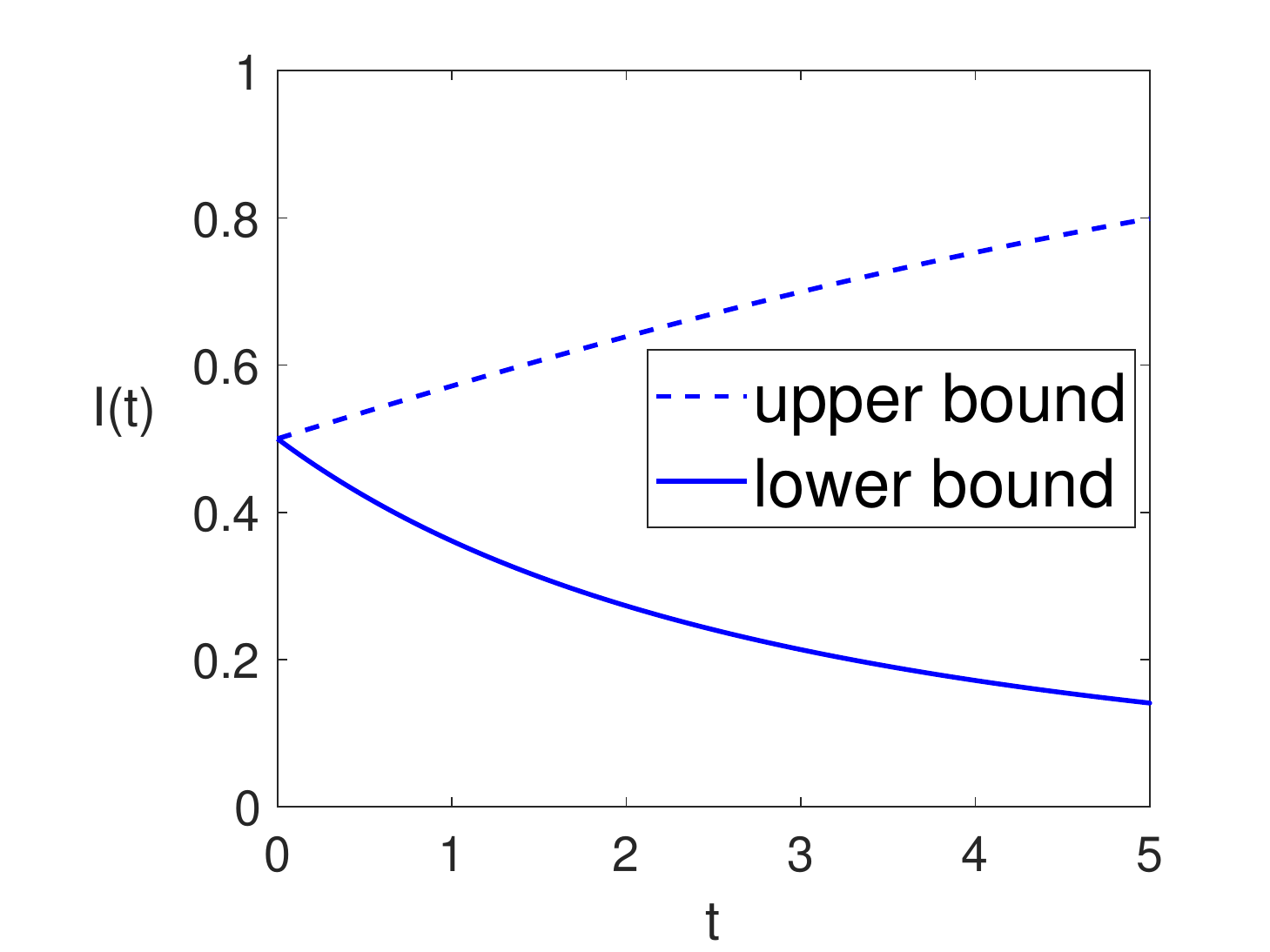}
		\subcaption{$p=0.5$}
	\end{subfigure}	
	\begin{subfigure}[h]{0.325\linewidth}
		\includegraphics[width=1\linewidth]{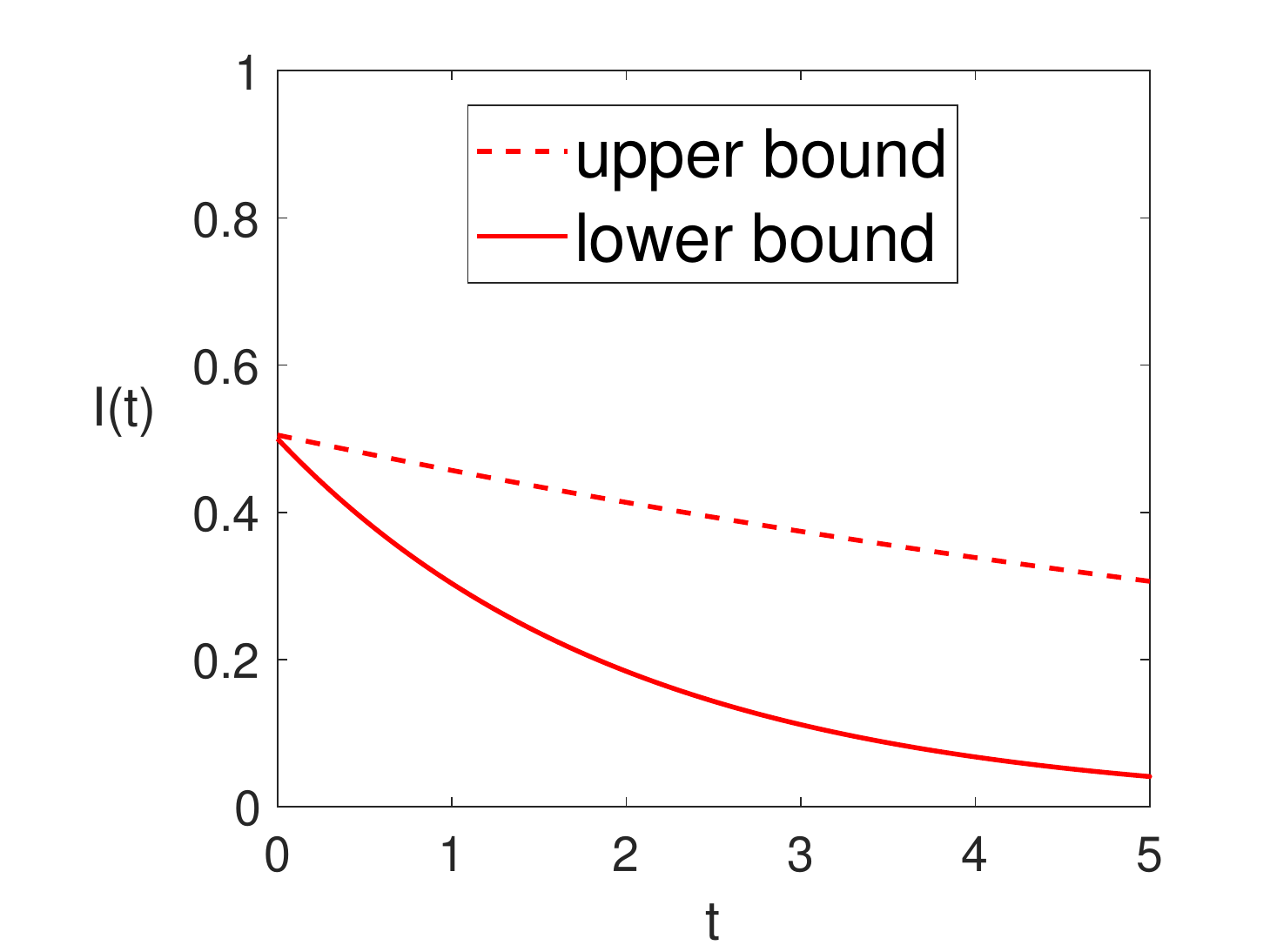}
		\subcaption{$p=0$}
	\end{subfigure}	

	\begin{subfigure}[h]{0.325\linewidth}
		\includegraphics[width=1\linewidth]{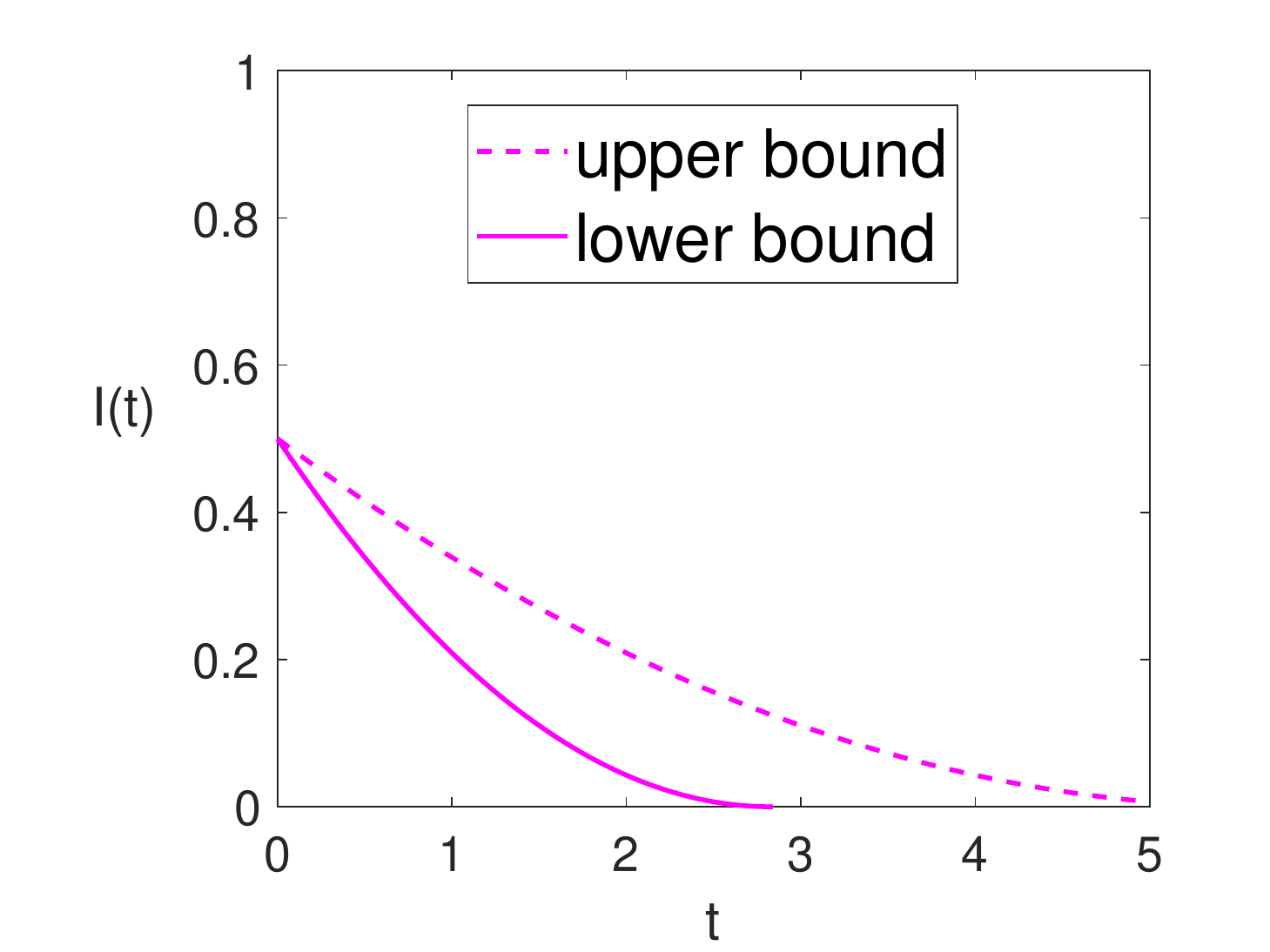}
		\subcaption{$p=-0.5$}
	\end{subfigure}	
	\begin{subfigure}[h]{0.325\linewidth}
		\includegraphics[width=1\linewidth]{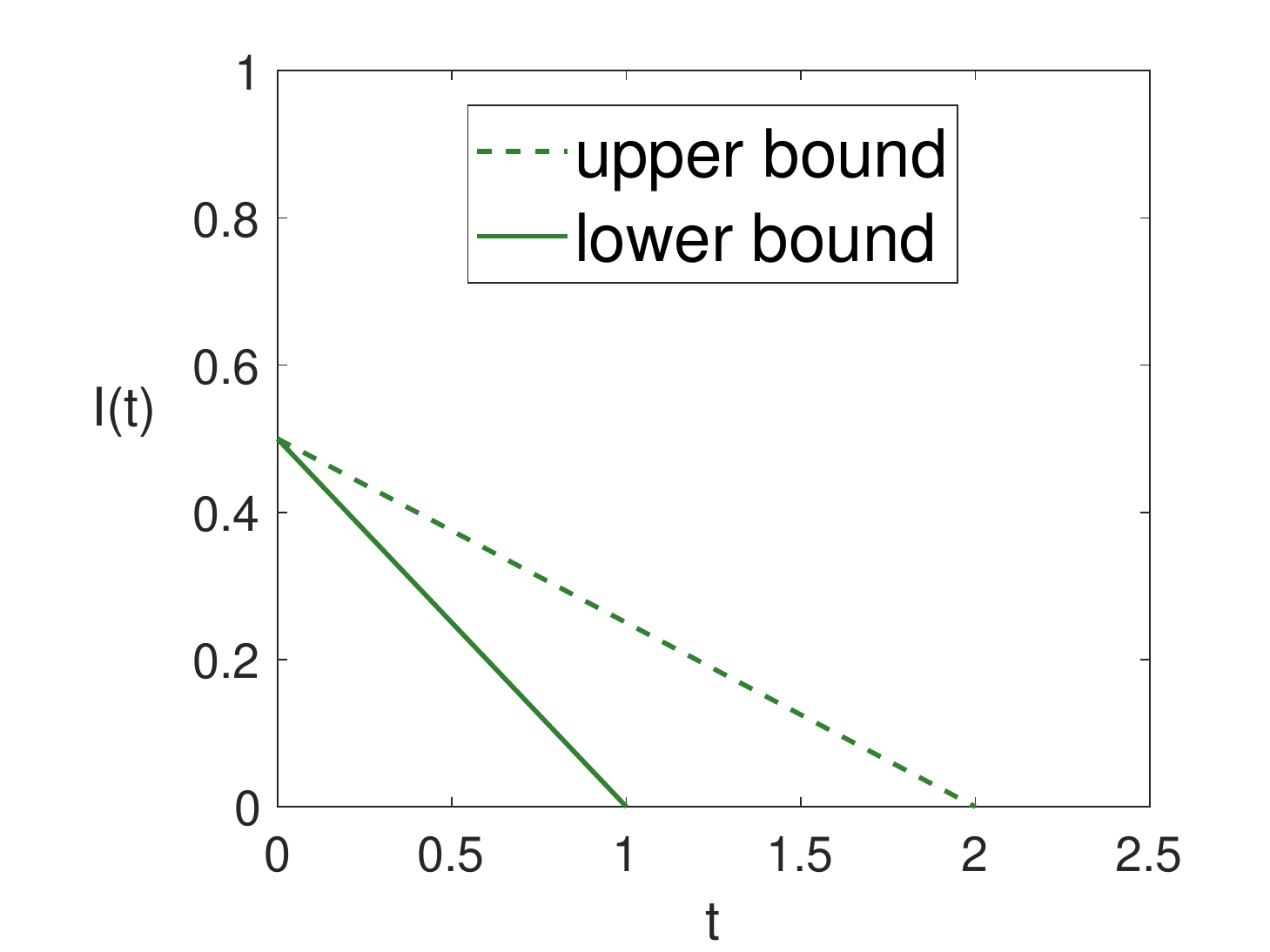}
		\subcaption{$p=-1$}
	\end{subfigure}	
	\begin{subfigure}[h]{0.325\linewidth}
		\includegraphics[width=1\linewidth]{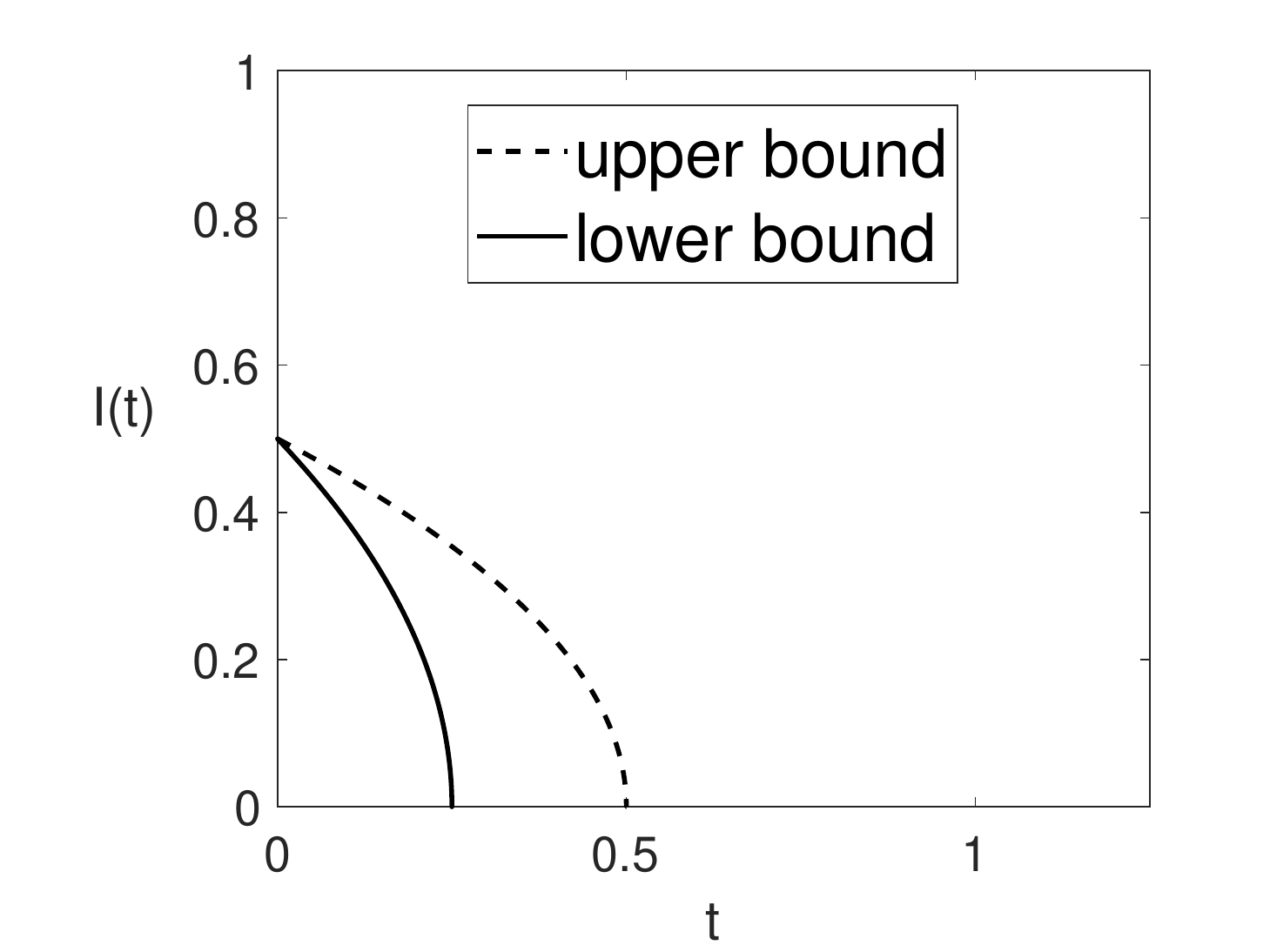}
		\subcaption{$p=-2$}
	\end{subfigure}	
	\caption{Description of upper and Lower bounds for $I(t)$ for different values of $p$, which corresponds to the estimates in Theorem \ref{main theorem1}. In each figure, we use the upper bounds $U(t)$ and lower bounds $L(t)$ defined as follows:\\
		(A) $U(t)=\left(\sqrt{2} -  \left(\sqrt{2} - 1 \right)\left(1-e^{-t/4}\right)\right)^{-2}$ and $L(t)=\left(t/4+\sqrt{2}\right)^{-2}$\\
		(B) $U(t)=\frac{1}{2}\exp(-0.1(t-0.1))$ and $L(t)=\frac{1}{2}\exp(-t/2)$\\
		(C) $U(t)= (\sqrt{2} -t/4)^2$ and
		$L(t)= (\sqrt{2} -t/2)^2$\\
		(D) $U(t)= 1/2-t/4$ and
		$L(t)= 1/2-t/2$\\
		(E)	$U(t)= ((1/2)^{2}-t/2)^{1/2}$ and $L(t)= ((1/2)^{2}-t)^{1/2}$
	}\label{fig estimate p0d5}
\end{figure}

\begin{proof}
$(1)$ From $0<S,I\leq 1$ and $\alpha\leq m_1$ we have
	\begin{align*}
		\begin{split}
			- C_*I^{p+1} \leq I'
			&\leq m_1 I- C_*I^{p+1}. 
		\end{split}
	\end{align*}
	Then, Lemma \ref{lem bernoulli p>0} gives
	\begin{align*}
		\begin{split}
			(I(t_*))^{-p}e^{-pm_1(t-t_*)} +  \frac{C_*}{m_1}\left(1-e^{-pm_1(t-t_*)}\right)\leq (I(t))^{-p}
			\leq pC_*(t-t_*)+(I(t_*))^{-p}.
		\end{split}
	\end{align*}
	Recalling the definition of $C_*=\beta(I(t_*))(I(t_*))^{-p}$ and the fact that $\beta(I(t_*))\leq m_3 < \frac{m_1S(0)}{2}< m_1$,
	we have
	\begin{align*}
		\begin{split}
			&(I(t_*))^{-p}e^{-pm_1(t-t_*)} +  \frac{C_*}{m_1}\left(1-e^{-pm_1(t-t_*)}\right)\cr 
			&= (I(t_*))^{-p} +  C_*\left(\frac{1}{m_1} - \frac{1}{\beta(I(t_*))} \right)\left(1-e^{-pm_1(t-t_*)}\right)
		\end{split}
	\end{align*}
	Then, we have
	\begin{align*}
		\begin{split}
			(I(t_*))^{-p} - \left(I(t_*))^{-p} - \frac{ C_*}{m_1} \right)\left(1-e^{-pm_1(t-t_*)}\right)	\leq (I(t))^{-p}
			\leq pC_*(t-t_*)+(I(t_*))^{-p}.
		\end{split}
	\end{align*}
	Rearranging this, we obtain the desired result.\newline
	
\noindent$(2)$	By Lemma \ref{Lem positivity}, for $p=0$ the following inequality holds 
\begin{align*}
	\begin{split}
		I'&= \alpha(I)SI- C_*I\geq - C_*I
	\end{split}
\end{align*}
due to the positivity of $S$, $\alpha$ and $I$. 
Thus, from Gr\"onwall's inequality, we obtain the desired lower bound.
We turn to the proof of the upper bound. 
By an explicit computation, we see that
$$I''(t_*)= \alpha(I(t_*))I(t_*)S'(t_*)<0.$$
Therefore, $I'(t)<0$ near $t_*$. Assume that there exists $t_{1}>t_*$ such that $I'(t_1)=0$ for the first time. This assumption implies
\begin{align}\label{I7}
	\begin{split}
		I(t_{1})&<I(t_*)
	\end{split}
\end{align}

On the other hand, from the fact that $S^{\prime}(t)=-\alpha(I) SI<0$, we get
\begin{align}\label{ISC}
	\begin{split}
		S(t_{1})<S(t_*).
	\end{split}
\end{align} 
Now, combining \eqref{I7}, \eqref{ISC} and the fact that $\alpha$ is an increasing function of $I$, we find 
$$0=I'(t_1)=(\alpha(I(t_{1}))S(t_{1})-C_*)I(t_1)<(\alpha(I(t_*))S(t_*)- C_*)I(t_*)=I^\prime(t_*)=0.$$
which is contradiction. Therefore, 
\[
I^\prime(t)<0\quad\mbox{ for all }t>t_*, 
\] 
That is, $I$ is strictly decreasing for $t>t_*$. This says that, if we take $t_{2}>t_*$ , then we have 
\begin{align*}
	\begin{split}
		I(t)&<I(t_2),\quad  S(t)<S(t_2)\quad \mbox{ for all }t>t_2
	\end{split}
\end{align*}
This leads to 
\[
\alpha(I(t))S(t)- C_*<\alpha(I(t_{2}))S(t_{2})- C_*.
\]
and
\begin{align*}
	\begin{split}
		I'(t)&=-(\alpha(I(t))S(t)- C_*) I(t)< -(\alpha(I(t_{2}))S(t_{2})- C_*) I(t),\quad t>t_{2}
	\end{split}
\end{align*}\newline

\noindent$(3)$	Since $0<S,I\leq 1$ and $\alpha\leq m_1$ we have from Lemma \ref{Prop small}
\begin{align}\label{01}
	\begin{split}
		- C_*I^{p+1}	\leq I'\leq m_1I- C_*I^{p+1}\leq -\frac{C_*}{2}I^{p+1}.
	\end{split}
\end{align}

This gives the desired estimate.	
\end{proof}

\section{Numerical Tests}
In this section, we provide numerical examples showing that our model \eqref{main model scale} is able to capture the cyclic nature of a trend, and its three different modes: Fad, Fashion and Classic.

\subsection{The influence of parameter $p$}
We show that $p$ determines the speed of decline of a trend, which enables us to characterize Fad, Fashion and Classic. To clarify the role of $p$, we fix other parameters $m_1$, $m_2$, $\ell_a$ and $\ell_b$ in the adoption/rejection rate $\alpha$, $\beta$ by
\begin{align}\label{init1_para}
	(m_1,m_2)&=(20,0.2),\, \ell_a=0,\quad 	(m_3,m_4)=(0.5,2),\, \ell_b=0
\end{align}
and set $\delta=0$. As initial data, we choose
\begin{align}\label{init1}
	S(0)=0.98,\quad I_1(0)=0.02,\quad R(0)=0.
\end{align}
In Figure \ref{fig m3 comparison}, we compare the profile of numerical solutions that correspond to various values of $p$. We observe that the choice $p> 0$ leads to a very slow relaxation tendency of $I$ (Classic). The case of $p=0$ shows a much faster decay rate (Fashion). When $-1<p<0$, $I$ touches zero in a finite time with a soft landing (Fast fashion). In case of $p\leq-1$, $I$ extincts in a finite time with a hard landing (Fad).

\begin{figure}[htbp]
	\centering
	\begin{subfigure}[h]{0.49\linewidth}
		\includegraphics[width=1\linewidth]{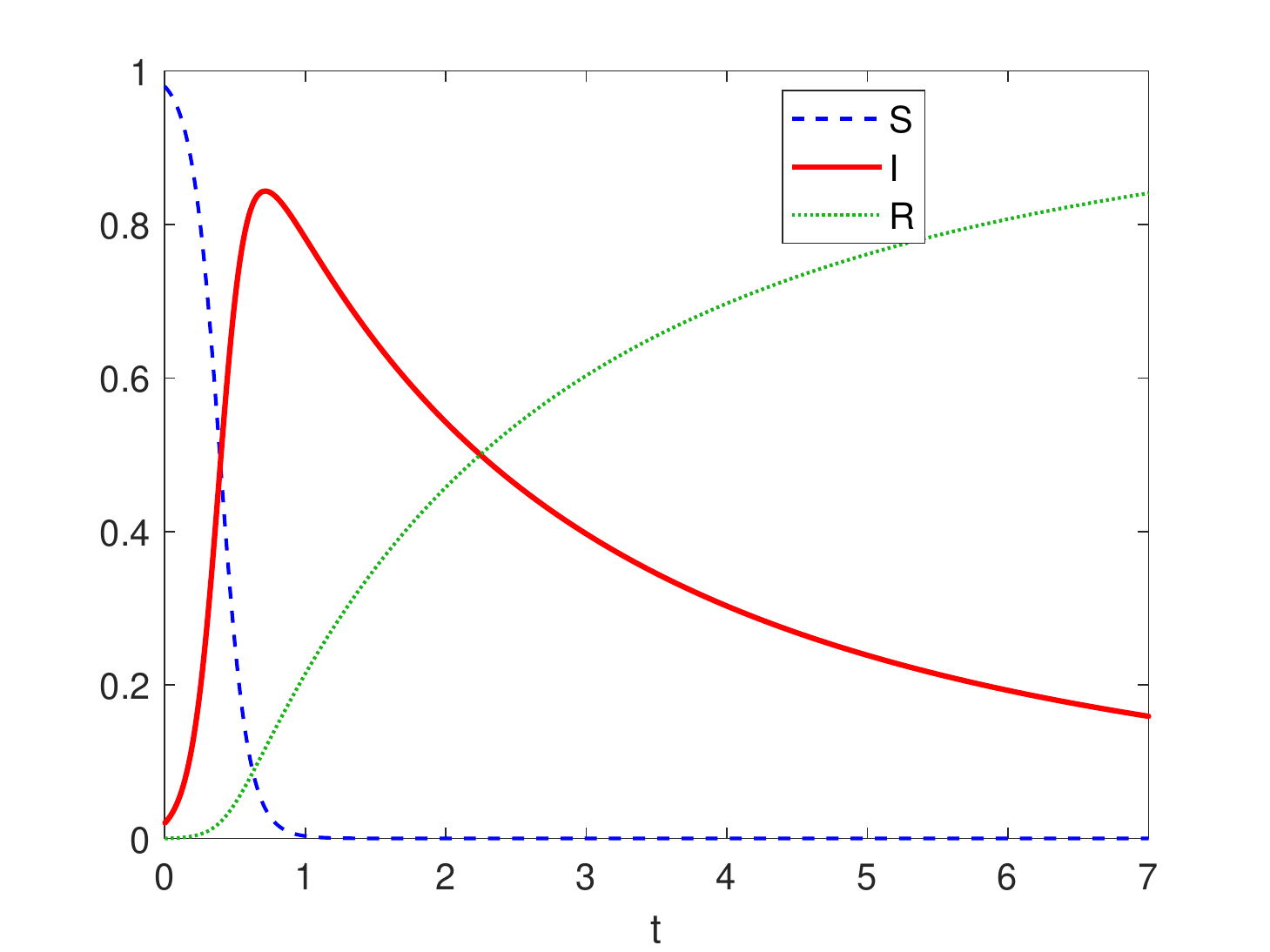}
		\subcaption{$p=0.5$}
	\end{subfigure}	
	\begin{subfigure}[h]{0.49\linewidth}
		\includegraphics[width=1\linewidth]{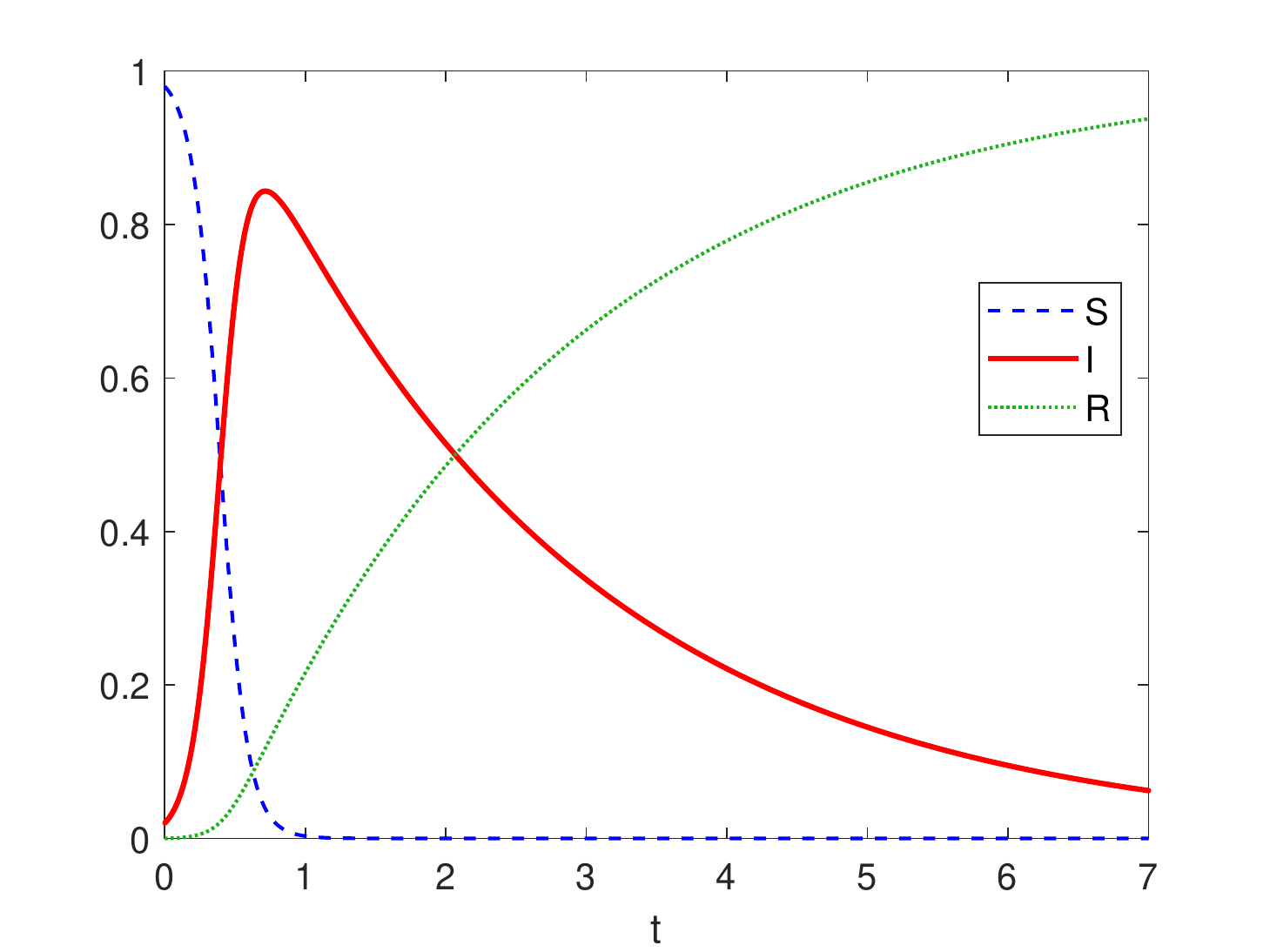}
		\subcaption{$p=0$}
	\end{subfigure}	
	\begin{subfigure}[h]{0.49\linewidth}
		\includegraphics[width=1\linewidth]{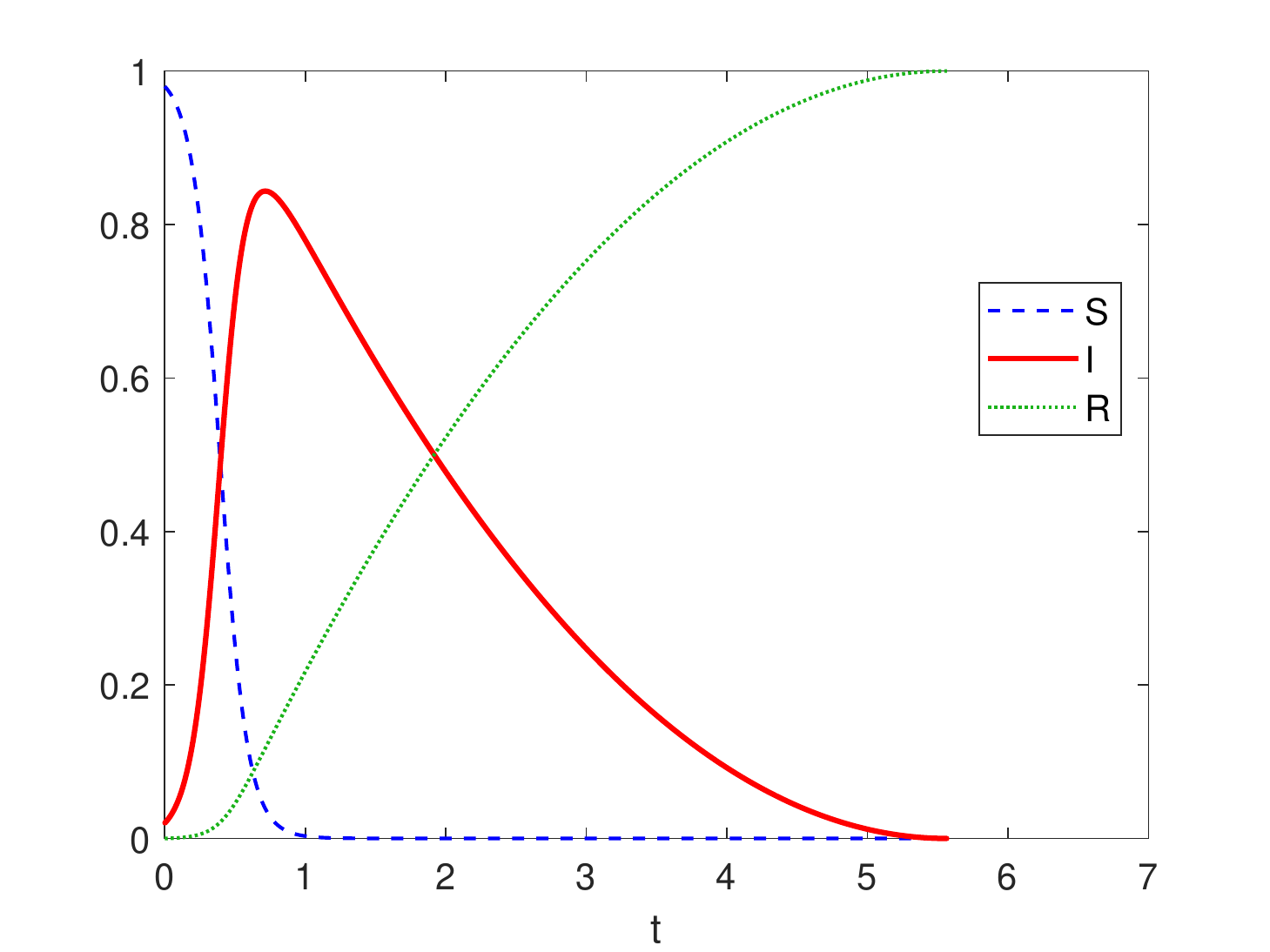}
		\subcaption{$p=-0.5$}
	\end{subfigure}	
	\begin{subfigure}[h]{0.49\linewidth}
		\includegraphics[width=1\linewidth]{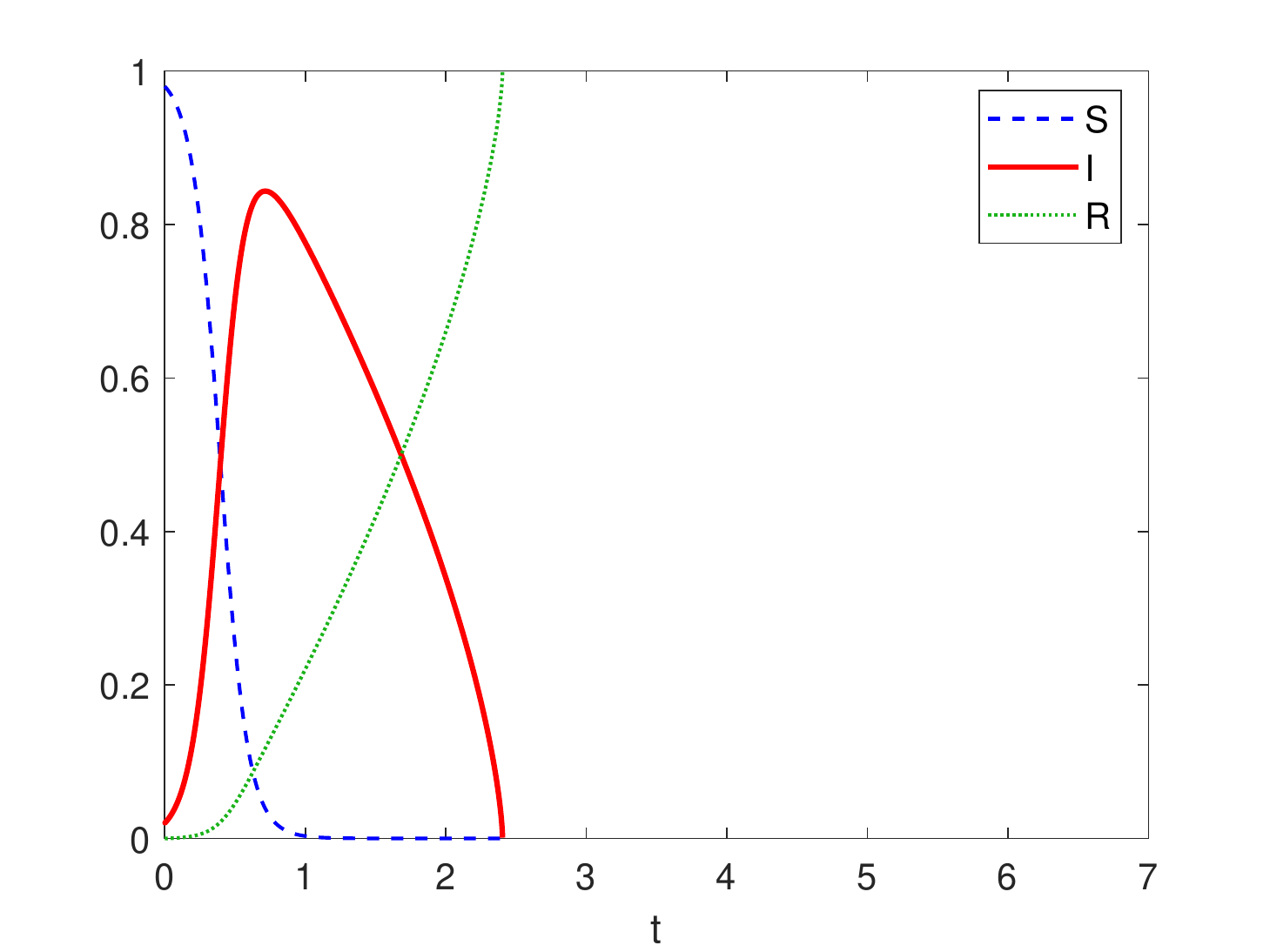}
		\subcaption{$p=-1.5$}
	\end{subfigure}	
	\caption{Time evolution of $S,I,R$.}\label{fig m3 comparison}\label{fig init0}
\end{figure}

\subsection{Expressive power of our model}
In this test, we show that our model has expressive power strong enough to reproduce various trend cycles through an appropriate choice of parameters in the adoption rate $\alpha$ and the rejection rate $\beta$. For this, we reproduce a diagram quoted from \cite{Lin} describing the trend cycle of Fad, Fashion and Classic (see Figure \ref{fig classification11}.). 
\begin{figure}[htbp]
	\includegraphics[width=0.6\linewidth]{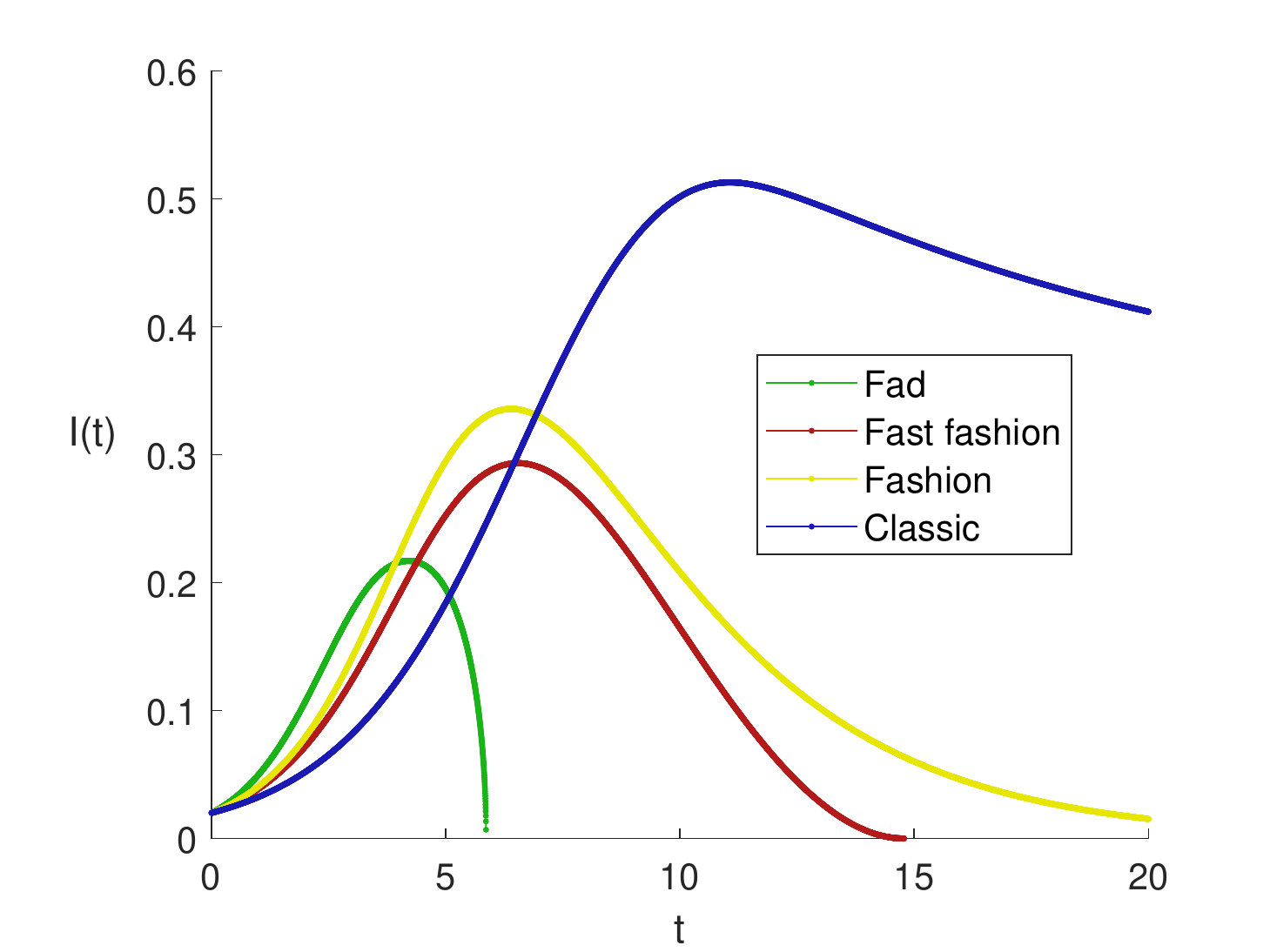}
	\caption{Reproduction of the trend cycles corresponding to Fad, Fashion, and Classic in Figure \ref{fig classification11}. We also add the trend cycle of Fast fashion. See text for parameter values.}\label{fig classification2}
\end{figure}
For each case, we use initial data 
\[
S(0)= 98,  ~I(0)= 2,  ~R(0)=0
\]
with the following set of parameters:
\begin{itemize}
	\item Fad:
	\[
m_1=3,\,    m_2=0.2,\,
m_3=1,\,       m_4=4,\,
\ell_\alpha=0,\,       \ell_\beta=0,\,
p=-2.
	\]
	\item Fast fashion:
	\[
	m_1=2,\,        m_2=0.03,\,
	m_3=0.6,\,        m_4=1.5,\,
	\ell_\alpha=0,\,       \ell_\beta=0,\,
	p=-0.5.
	\]
	\item Fashion:
	\[
	m_1=2,\,       m_2=0.05,\,
	m_3=0.5,\,        m_4=2\,,
	\ell_\alpha=0,\,       \ell_\beta=0,\,
	p=0.
	\]
	\item Classic:
	\[
	m_1=1.2,\,       m_2=0.03,\,
	m_3=0.2,\,        m_4=0.5,\,
	\ell_\alpha=0,\,       \ell_\beta=0,\,
	p=6.
	\]
\end{itemize}

\subsection{Recurring cycle of a trend}
By taking positive value of $\delta$, our model can also describe a recurring cycle of the trend. The case $\delta > 0$ corresponds to the situation where consumers (who have left a trend once) become potential consumers again. In order to clarify the role of $\delta$ in determining the scale and period of such a cycle, throughout this test we fix $$(m_1,m_2,m_3,m_4)=(50,8,4,0.5),\, p=0,\,\ell_\alpha=0.3,\,\ell_\beta=0$$ and take different types of $\delta$. Here, we use the same initial data \eqref{init1}.
In Figure \ref{fig c1}, we observe that the choice of $\delta$ allows us to describe the recurring behavior of numerical solutions.
The simulation shows that a time-dependent function $\delta(t)$ describes much sharper but less fluctuating solutions compared to the case for fixed $\delta$. 
We can take the recurrence rate $\delta$ to be a time-dependent function to describe 
a more complicated periodic trend cycle (Figure \ref{fig c1} (B)). 

\begin{figure}[!htbp]
	\centering
	\begin{subfigure}[h]{0.49\linewidth}
		\includegraphics[width=1\linewidth]{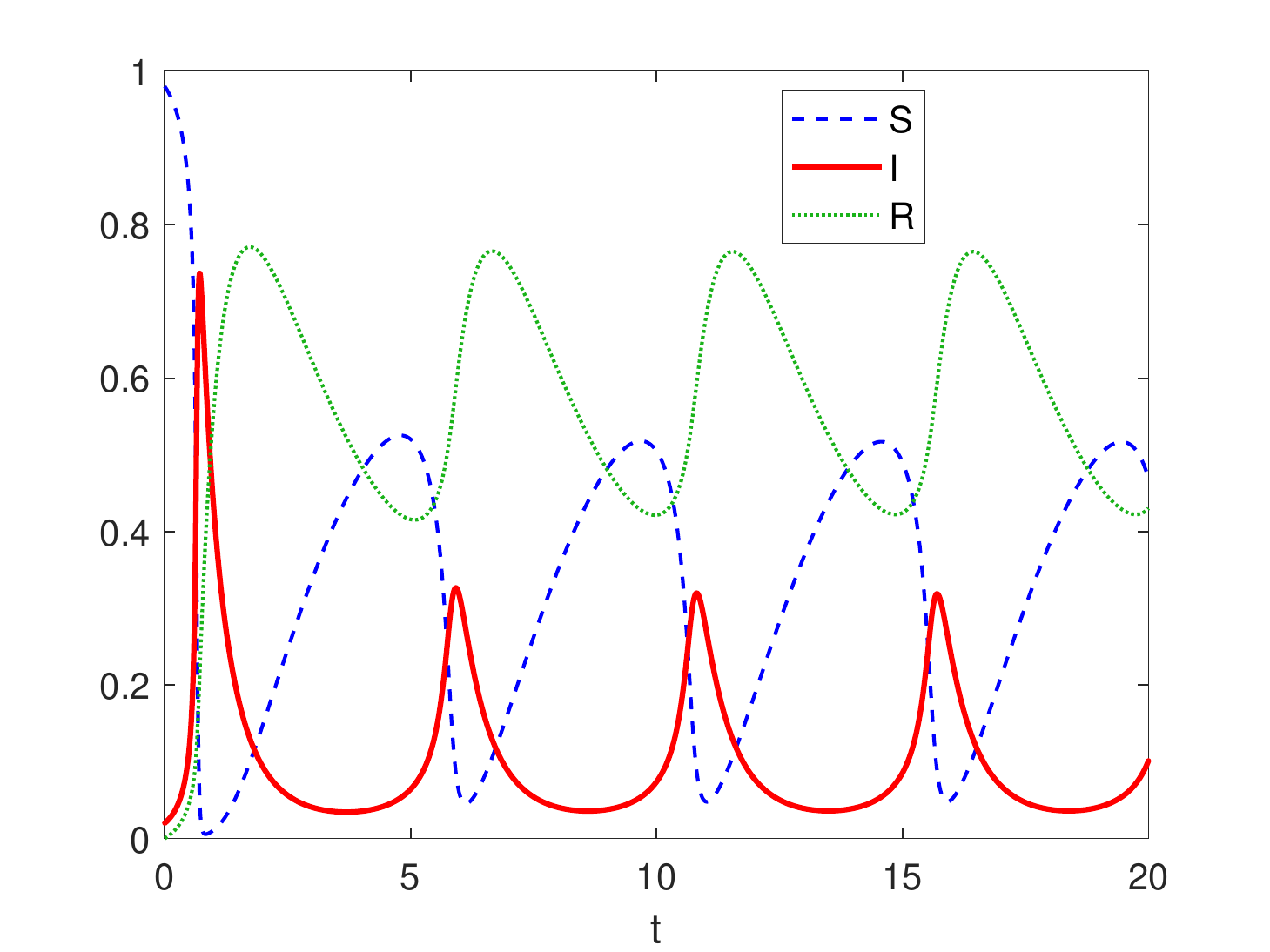}
		\subcaption{$\delta=0.4$}
	\end{subfigure}	
\begin{subfigure}[h]{0.49\linewidth}
	\includegraphics[width=1\linewidth]{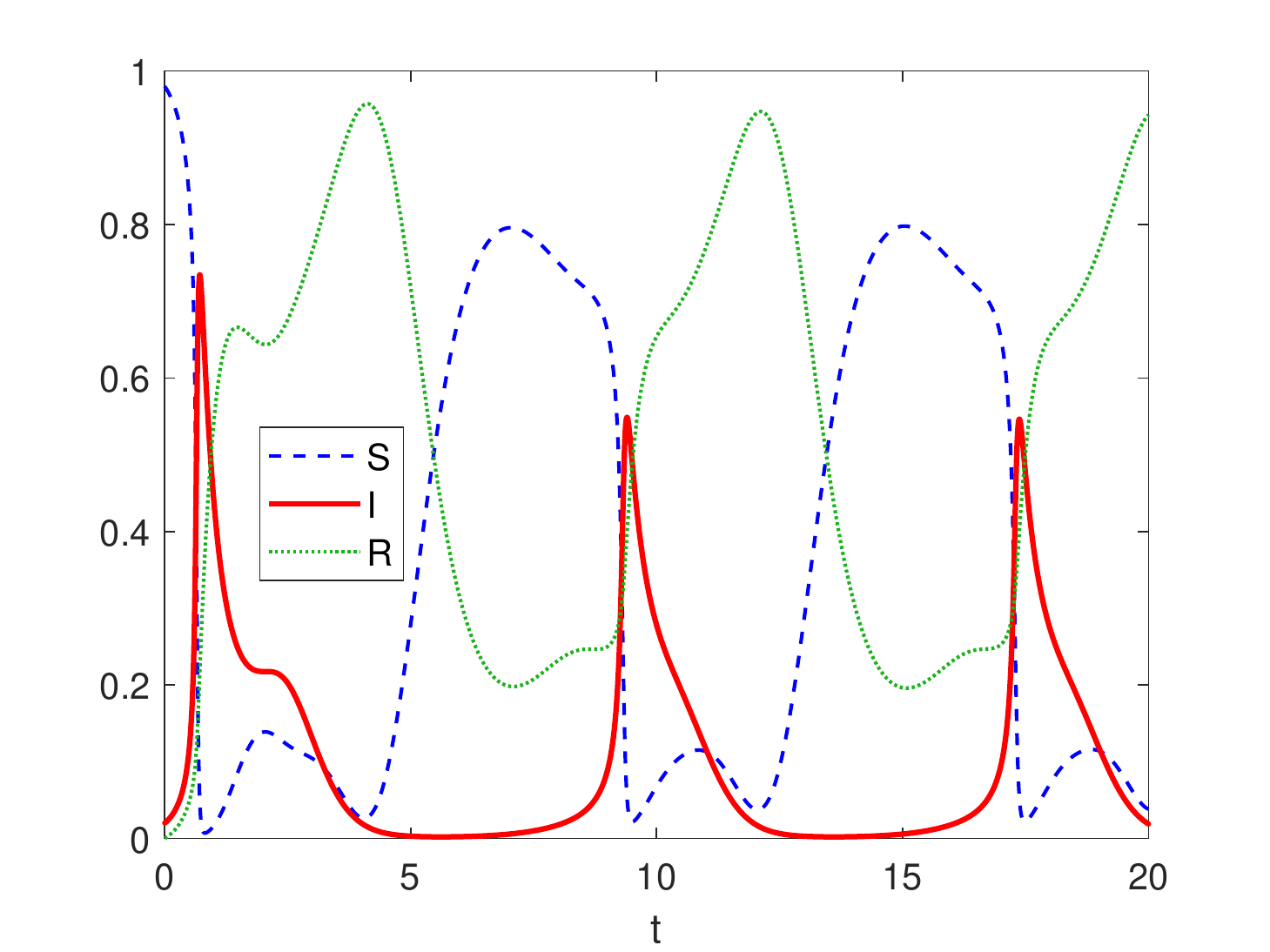}
	\subcaption{$\delta(t)=0.4+0.5\sin(\frac{\pi}{2}t-1)$}
\end{subfigure}	
	\caption{Recurrence of trend cycles depending on $\delta$.}\label{fig c1}
\end{figure}

\section{Conclusion} 
In this work, we suggest a system of differential equations that quantitatively model the evolution of a trend through the consideration of underlying dynamics between trend participants. Our model captures five stages of a trend cycle, namely, the onset, rise, peak, decline, and obsolescence. It also provides a mathematical criterion to divide Fad, Fashion and Classic. We also prove various mathematical properties of our model, and provide simulations to justify the modeling. 
Note that we have investigated a mono-fashion model, where only one trend is considered for simplicity.
Extension of such mono-fashion model into a multi-fashion model, where the differentiation into a sub-fashion is taken into consideration, will
be an interesting topic. This is our upcoming work.\\

\noindent{\bf Acknowledgement:}

Bae is supported by the Basic Research Program through the National Research Foundation of Korea
(NRF) funded by the Ministry of Education and Technology (NRF-2021R1A2C1093383), Cho (RS-2022-00166144), Yoo by Ajou University Research Fund, and Yun by Samsung Science and Technology Foundation under Project Number SSTF-BA1801-02.

\appendix
\section{Proof of Lemma 3.2}
Since $S(t)=S(0)e^{-\int^t_0\alpha I(s)ds}$ and $R(t)=\beta SI$, it is enough to show that $I(t)>0$. 
\noindent $(1)$ Case of $t<t_*$: In this case, we have	
\[
I'(t)\geq -\beta(I)I\geq -m_3I.
\]
Therefore
\[
I(t)\geq\alpha S I-\beta I\geq -\beta I\geq e^{-m_3t}I(0)>0.
\]
This gives the desired result.\newline
\noindent $(2)$ Case of  $t>t_*$ and $p\geq0$:  In this case, we use $I^{p+1}\leq I$ to get
\[
I'(t)=\alpha S I-C_* I^{p+1}\geq -C_*I^{p+1}\geq -C_*I(t)
\]
so that $I(t)\geq I(t_*)e^{-C_*(t-t_*)}>0.$\newline
\noindent $(3)$ Case of  $t_*<t<T_*$ and $p<0$: In this case, $I$ satisfies
\[
I'(t)= \alpha S I-C_* I^{p+1}\geq-C_*I^{p+1}
\] 
which yields:
\[
I(t)\geq \left(I^{-p}(t_*) + pm_3(t-t_*)\right)^{-\frac{1}{p}}>0.
\]

\section{Proof of Lemma 3.3}
Since $\displaystyle m_1S(0)>2m_3$, $I$ increases initially:
\begin{align*}
	\begin{split}
		I'(0)&=\alpha(I(0)) I(0) S(0) - \beta(I(0)) I(0)
		\geq \frac{m_1}{2} I(0) S(0) - m_3 I(0)
		\geq \left(\frac{m_1}{2} S(0) - m_3\right) I(0)
		>0.
	\end{split}
\end{align*}
Assume $I'(t)>0$ for all $t$ so that
\begin{align*}
	\begin{split}
		I(t)&\geq I(0)>0\quad\mbox{ for }t>0
	\end{split}
\end{align*}
However, this implies that
$$S(t)\leq S(0)e^{-\int_0^t \alpha I(0) ds}\leq S(0)e^{-\frac{m_1}{2} I(0)t} \to 0$$
as $t$ goes to $\infty$. That is, there is  a sufficiently large time $T>0$ such that  $$S(t)<\frac{m_3}{4m_1}<\frac{1}{8}$$ for $t>T$. Hence
\begin{align*}
	\begin{split}
		I'&\leq (m_3/4) I - (m_3/2)I  <0.
	\end{split}
\end{align*}
This contradicts the assumption. Therefore, $t_*$ is finite.
\section{Proof of Lemma 3.4}
We first make the following claim:\newline
\textbf{Claim:} Choose any $I_*$ such that $0<I^*<I(t_*)$, then $I(t)$ becomes smaller than $I^*$ in a finite time. \newline
Proof of the claim: To prove the claim, suppose contrarily that $I^*\leq I(t)\leq 1$ for all $t>t_*$. 
Under this assumption, $S$ satisfies
\[
S(t)=S(t_*)e^{-\int_0^t \alpha(I(s))I(s) ds}
\leq S(t_*)e^{-\int_{t_*}^t \frac{m_1}{2} I(s)ds}\leq S(t_*)e^{-\frac{m_1}{2} I_*(t-t_*)}.
\]
Therefore, for sufficiently large $t$, we have
\begin{align}
	\begin{split}
		I'&=\alpha(I) SI- C_* I^{p+1}\leq \big(m_1S-C_*\big)I\leq -  \frac{C_*}{2} I.
	\end{split}
\end{align}
This implies that $I$ decreases exponentially fast, which contradicts the assumption that $I(t)\geq I^*>0$ for all $t>t_*$. This proves the claim.\newline

Now, from the definition of $C_*$ given in the definition of $\beta$, we see that
 $$\left(\frac{C_*}{2m_1}\right)^{-1/p} \leq \left(\frac{m_3}{2m_1} \left(I(t_*)\right)^{-p}\right)^{-1/p}\leq I(t_*)$$
 In the last inequality, we used $\displaystyle m_3/m_1<S(0)/2<1$.
 Therefore, thanks to the claim, we can find $\tau>t_*$ such that
 $$I(\tau)= \left(\frac{C_*}{2m_1}\right)^{-\frac{1}{p}}.$$
Let $\bar t>\tau$ be the first time such that $I(\bar t)=I(\tau)$. Then, 
$I(t)< I(\tau)$ holds for $\tau < t < \bar{t}$. Hence
\begin{align*}
	\begin{split}
		I'(t)\leq  m_1I(t)- C_* I(t)^{p+1}<  m_1I(t)- C_* \{I(\tau)\}^pI(t)= m_1I(t)-\left(2m_1/C_*\right)I(t)\leq -m_1I(t),
	\end{split}
\end{align*}
which implies $I(\bar{t})$ is strictly less than $I(\tau)$ on  $\tau < t < \bar{t}$. This contracts the assumption.
Therefore, for all $t\geq\tau$ we have $I(t)\leq I(\tau)$, which is equivalent to $$I(t)\leq \frac{C_*}{2m_1}\left(I(t)\right)^{p+1}.$$

\section{Proof of Lemma 3.5}
Dividing both sides of the inequality by $I^{p+1}$, we get
\begin{align*}
	\begin{split}
		-b \leq \left(\frac{1}{-p}I^{-p}\right)'&\leq aI^{-p}- b.
	\end{split}
\end{align*}
The second inequality is trivial. For the first one, we apply Gr\"onwall's lemma to obtain
\begin{align*}
	\begin{split}
		I^{-p}(t) &\geq I^{-p}(T_0)e^{-pa(t-T_0)} +  \frac{b}{a}\left(1-e^{-pa(t-T_0)}\right),
	\end{split}
\end{align*}
which completes the proof.

\end{document}